\documentclass[preprint,12pt]{elsarticle}
\usepackage{graphicx}
\usepackage{amsmath, amssymb, amsthm}
\usepackage{xcolor}
\usepackage{hyperref}
\usepackage{cleveref}
\usepackage{caption}
\usepackage{subcaption}
\usepackage{comment}

\newtheorem{theorem}{Theorem}
\newtheorem*{remark}{Remark}

\begin{document}

\begin{frontmatter}
\title{
Separable Physics-informed Neural Networks for Solving the BGK Model of the Boltzmann Equation.
}

\author[label1]{Jaemin Oh}
\ead{jaemin_oh@kaist.ac.kr}
\affiliation[label1]{
organization={Department of Mathematical Sciences, Korea Advanced Institute of Science and Technology (KAIST)},
city={Daejeon},
postcode={34141},
country={Republic of Korea}
}

\author[label2]{Seung Yeon Cho}
\ead{chosy89@gnu.ac.kr}
\affiliation[label2]{
organization={Department of Mathematics, Gyeongsang National University},
city={Jinju},
postcode={52828},
country={Republic of Korea}
}

\author[label3]{Seok-Bae Yun}
\ead{sbyun01@skku.edu}
\affiliation[label3]{
organization={Department of Mathematics, Sungkyunkwan University},
city={Suwon},
postcode={16419},
country={Republic of Korea}
}

\author[label4]{Eunbyung Park}
\ead{epark@skku.edu}
\affiliation[label4]{
organization={Department of Electrical and Computer Engineering, Sungkyunkwan University},
city={Suwon},
postcode={16419},
country={Republic of Korea}
}

\author[label1]{Youngjoon Hong}
\ead{hongyj@kaist.ac.kr}

\begin{abstract}
In this study, we introduce a method based on Separable Physics-Informed Neural Networks (SPINNs) for effectively solving the BGK model of the Boltzmann equation.
While the mesh-free nature of PINNs offers significant advantages in handling high-dimensional partial differential equations (PDEs), challenges arise when applying quadrature rules for accurate integral evaluation in the BGK operator, which can compromise the mesh-free benefit and increase computational costs.
To address this, we leverage the canonical polyadic decomposition structure of SPINNs and the linear nature of moment calculation, achieving a substantial reduction in computational expense for quadrature rule application.
The multi-scale nature of the particle density function poses difficulties in precisely approximating macroscopic moments using neural networks.
To improve SPINN training, we introduce the integration of Gaussian functions into SPINNs, coupled with a relative loss approach.
This modification enables SPINNs to decay as rapidly as Maxwellian distributions, thereby enhancing the accuracy of macroscopic moment approximations.
The relative loss design further ensures that both large and small-scale features are effectively captured by the SPINNs.
The efficacy of our approach is demonstrated through a series of five numerical experiments, including the solution to a challenging 3D Riemann problem. These results highlight the potential of our novel method in efficiently and accurately addressing complex challenges in computational physics.
\end{abstract}

\end{frontmatter}

\section{Introduction}
The Boltzmann equation fundamentally characterizes the temporal evolution of particle density functions, predicated on the binary collision model.
This equation diverges from traditional fluid dynamics equations by its capacity to encapsulate dynamics that extend beyond the continuum regime, offering a broader applicative scope \cite{villani2002review, lewis1984computational, markowich2012semiconductor, hazeltine2018framework}.
However, the practical utility of the Boltzmann equation is significantly constrained by the computational intensity inherent in its high-dimensional collision operator.
This complexity imposes substantial computational demands, thereby limiting its application across a diverse spectrum of scientific fields \cite{mieussens2000discrete, filbet2010class, dimarco2014numerical, romano2021openbte}.

Much effort has been directed towards the development of numerical methods for simulating kinetic equations. Among these, the Direct Simulation Monte Carlo (DSMC) method \cite{bird1994molecular}, employing a stochastic framework for the direct resolution of the Boltzmann equation, is distinguished by its computational efficiency.
Nevertheless, the accuracy of DSMC is somewhat undermined by the inherent presence of statistical noise and pronounced oscillations.
In contrast, deterministic methodologies, exemplified by the Fourier spectral method \cite{pareschi1996fourier,pareschi2000numerical,mouhot2006fast, liu2024convergence} and the discrete velocity model \cite{mieussens2000discrete}, demonstrate superior accuracy, albeit with a concomitant increase in computational load.
A notable advancement in this domain is documented in \cite{dimarco2013towards, dimarco2018efficient}, where the BGK model of the Boltzmann equation and the full Boltzmann equation were efficiently resolved within a three-dimensional spatial domain, albeit necessitating extensive parallelization of computational resources. 
This highlights the complexity inherent in the numerical simulation of kinetic equations, especially in devising feasible numerical schemes that simultaneously reduce computational costs.
Consequently, the quest for an efficient and practical numerical method emerges as a formidable challenge, requiring rigorous theoretical development and strategic application approaches.

Recent developments have led to the application of neural networks in reducing computational complexities associated with kinetic equation simulations.
These efforts are broadly divided into two categories.
The first is a data-driven approach \cite{porteous2021data, alekseenko2022fast, xiao2023relaxnet}, akin to reduced-order models \cite{tsai2023accelerating}.
In this method, neural networks, trained with high-fidelity numerical solutions (data), are used to replace the most computationally demanding elements of the simulation.
For example, Han et al. \cite{han2019uniformly} utilized neural networks to tackle the moment equation by learning the moment closure relation.
In a similar vein, Miller et al. \cite{miller2022neural} introduced a pre-trained surrogate neural network for the entire Boltzmann collision operator, aiming to reduce computational expenses.
However, these methods require a substantial number of input-output pairs for an accurate approximation of the true relationship, posing challenges in situations where data collection is difficult or expensive.

The second approach to reducing computational complexity is the data-free method, with Physics-Informed Neural Networks (PINNs) \cite{raissi2019physics, lu2021physics, lee2023oppinn, muller2023achieving, wang2023expert, zhou2023physics} being a prominent example.
PINNs model the solution of partial differential equations (PDEs) using neural networks, integrating PDE residuals, initial conditions, and boundary conditions as penalty terms.
The inherent advantage of PINNs, not requiring structured grids, positions them as a promising tool for efficiently solving high-dimensional PDEs. 
A significant advancement in this domain is demonstrated in \cite{hu2023tackling, hu2023hutchinson}, where PDEs of extremely high dimensions (up to $10^5 d$) were solved by introducing stochastic elements in the dimensions for gradient descent, extending beyond collocation points.
However, this method did not encompass integro-differential equations (IDEs), limiting its applicability to the BGK model.
A notable challenge with IDEs is that the use of quadrature rules for integral evaluation partially compromises the grid-free nature of PINNs, involving a substantially higher number of network forward passes.

In this study, we introduce ``SPINN-BGK'', a novel machine-learning framework to effectively solve the BGK model.
To overcome the challenges previously outlined, we have adopted the framework of Separable PINNs \cite{cho2024separable}, distinguished by their canonical polyadic decomposition structure.
This methodology, combined with a novel integration strategy that alters the conventional order of integrals and summation, markedly diminishes both memory and computational overheads.
Consequently, the application of Separable PINNs in conjunction with this integration technique has enabled the effective resolution of the BGK model in a three-dimensional spatial domain on a single GPU. However, computational complexity is not the sole challenge.
PINNs have shown limitations in accurately approximating macroscopic quantities such as density, velocity, and temperature.
This issue arises, in part, due to the characteristics of neural networks, which, without specific modifications, fail to replicate the rapid decay in the velocity domain exhibited by the particle density function.
To rectify this, we have integrated Gaussian functions of microscopic velocity into the neural network architecture.
Furthermore, we utilized a relative $L^2$ loss function to ensure a balanced approximation of features across varying magnitudes.
Additionally, we decomposed the neural network into two components: an equilibrium part representing the Maxwellian distribution, and a non-equilibrium part, as in \cite{bennoune2008uniformly,jin2010micro}.
This decomposition is inspired by the BGK operator's role in driving the system towards local thermodynamic equilibrium.
Collectively, these strategies have enabled our proposed neural networks to rapidly decay along the velocity domain, effectively approximate features of diverse magnitudes, and expedite training, thereby yielding accurate macroscopic moments.

Diverse machine learning methodologies have recently been applied to tackle the complexities of the BGK model.
Lou et al. \cite{lou2021physics} investigated both forward and inverse problems within the framework of PINNs, specifically targeting flows across a range of Knudsen numbers. However, their approach was primarily limited to the lattice Boltzmann method, employing a comparatively coarse lattice for the microscopic velocity space.
Another notable attempt by Li et al. \cite{li2023solving} involved the application of canonical polyadic decomposition to the discretized microscopic velocity space, followed by solving the full Boltzmann equation using a singular value decomposition (SVD)-based reduced-order formulation, underpinned by an ansatz derived from the solution to the BGK model. However, their focus was primarily on the 2D Boltzmann equation.
Our research diverges by concentrating on the BGK model, which is inherently less computationally demanding, enabling us to present a three-dimensional example.
Moreover, by applying tensor decomposition to the spatio-temporal domain, SPINN-BGK achieves enhanced computational efficiency, particularly for the BGK equation.
Beyond these computational advancements, we also elucidate the potential shortcomings of the PINNs method in accurately generating macroscopic moments and propose a strategy to address these limitations.

The structure of the remaining sections of this paper is organized as follows: \Cref{sec:preliminaries} provides a concise overview of the BGK model and the framework of SPINNs. In \Cref{sec:methodology}, we delve into the detailed exposition of our proposed methodologies.  \Cref{sec:numerical_results} is dedicated to presenting the results derived from both smooth and Riemann problem scenarios. Finally, \Cref{sec:conclusion} offers a comprehensive summary and conclusion of our research findings.

\section{Preliminaries}\label{sec:preliminaries}
\subsection{The BGK Model of the Boltzmann Equation}
Consider a particle characterized by a spatial dimension $d \in \{1,2,3\}$.
At any given time $t$, the position of this particle is denoted as ${\bf x} \in \Omega$, where $\Omega$ is a subset of the Euclidean space $\mathbb{R}^d$.
The particle is capable of moving with a velocity ${\bf v}$, which is an element of $\mathbb{R}^3$. We define $f(t, {\bf x, v})$ as the particle density function in the phase space $\Omega \times \mathbb{R}^3$, representing the distribution of particles at time $t$.

The Boltzmann equation characterizes the time evolution of the particle density function $f$, based on the binary collision assumption.
A significant aspect of this equation is its collision integral, which involves a five-fold computation, often leading to substantial computational demands \cite{dimarco2014numerical}. To simplify this, the binary collision integral is replaced with a relaxation process that approximates the movement towards local thermodynamic equilibrium.
This adaptation results in the Bhatnagar-Gross-Krook (BGK) model \cite{bhatnagar1954model}, which is formulated as follows: 
\begin{equation}\label{eq:bgk}
    \frac{\partial f}{\partial t} + {\bf v} \cdot \nabla_{\bf x} f 
    = \frac{1}{\mathrm{Kn}} \left( M[f] - f\right).
\end{equation}
The Knudsen number, denoted as $\mathrm{Kn}$, serves as a dimensionless parameter, representing the ratio of the mean free path of a particle to the physical length scale of interest. In this context, we assume a fixed collision frequency. The Maxwellian distribution, denoted as $M[f]$, is defined as the particle density at local thermodynamic equilibrium. This distribution is characterized by
\[
    M[f](t,{\bf x,v}) = \frac{\rho(t,{\bf x})}{\left(2\pi T(t,{\bf x})\right)^{3/2}}e^{-\frac{|{\bf v} - {\bf u}(t,{\bf x})|^2}{2T(t,{\bf x})}},
\]
where the macroscopic mass $\rho$, velocity ${\bf u}$, and temperature $T$ are defined by
\begin{align*}
    &\rho(t,{\bf x}) = \int f(t,{\bf x, v}) d{\bf v}, \\
    &{\bf u}(t,{\bf x})
    =
    \begin{pmatrix}
      u_x(t,{\bf x}) \\ u_y(t,{\bf x}) \\ u_z(t,{\bf x})  
    \end{pmatrix}
    = \frac{1}{\rho(t,{\bf x})}\int
    \begin{pmatrix}
    v_x \\ v_y \\ v_z    
    \end{pmatrix}
    f(t,{\bf x, v}) d{\bf v}, \\
    &T(t,{\bf x})
    = \frac{1}{3}\left(
    \frac{1}{\rho(t,{\bf x})}\int |{\bf v}|^2 f(t,{\bf x, v}) d{\bf v}- 
    |{\bf u }(t,{\bf x})|^2 
    \right).
\end{align*}
Note that all the integrals mentioned above involve triple integration over $\mathbb{R}^3$.
As a result, although there are various numerical methods available for solving the BGK model \cite{mieussens2000discrete,bennoune2008uniformly,pieraccini2007implicit,cho2021conservative2,ding2021semi,xiong2015high,boscheri2020high, gamba2019micro}, they typically face high computational costs.
This challenge arises mainly from the model's high dimensionality and the necessary integrals for calculating the macroscopic moments.

\subsection{Physics-informed Neural Networks}
Consider $\phi : \mathbb{R}^{d_\mathrm{in}} \rightarrow \mathbb{R}^{d_\mathrm{out}}$ as a feed-forward neural network (FNN) comprising $L$ layers, with $d_l$ denoting the number of neurons in the $l$-th layer.
We define $d_0 = d_\mathrm{in}$ and $d_L = d_\mathrm{out}$.
The affine transformation $A_l$ for the $l$-th layer is characterized by a weight matrix $W_l \in \mathbb{R}^{d_l \times d_{l-1}}$ and a bias vector $b_l \in \mathbb{R}^{d_l}$.
Let $\theta := \{ W_l, b_l: 1 \le l \le L\}$ represent the set of network parameters. With an activation function $\sigma(\cdot)$ applied element-wise, the feed-forward neural network, parameterized by $\theta$, is defined as
\[
    \phi(\cdot) = A_L \circ \sigma \circ A_{L-1} \circ \sigma \circ \cdots \circ A_1 (\cdot).
\]
In this context, $\phi$ is often denoted as $\phi_\theta$ to emphasize its dependency on the network parameters $\theta$.

A primary objective of PINNs is to approximate the solution of a PDE using a neural network.
This task essentially translates to identifying an optimal set of network parameters, denoted as $\theta$, such that the neural network $\phi_\theta$ accurately satisfies the stipulations of the given PDE.
To elucidate the conventional methodology for determining these network parameters within the PINN framework, we consider the BGK model in the following specific form: 
\begin{align}
    \mathcal{L}_\mathrm{BGK}[f](t,{\bf x,v}) &= 0, \quad \forall (t,{\bf x,v}) \in (0,\infty)\times \Omega \times \mathbb{R}^3, \label{eq:bgk1} \\
    \mathcal{B}[f](t,{\bf x,v}) &= 0, \quad \forall (t, {\bf x, v}) \in (0, \infty) \times \partial \Omega \times \mathbb{R}^3, \label{eq:bgk_bc}
\end{align}
where $\mathcal{L}_\mathrm{BGK}[\cdot]$ is the integro-differential operator of the BGK model and $\mathcal{B}[\cdot]$ may encompass Dirichlet, Neumann, or periodic boundary conditions depending on the specific requirements of the problem at hand.
Note that this setup is not limited to the BGK model.
Given an activation function $\sigma$, a specified boundary condition $\mathcal{B}$, and an initial condition $f_0$, the objective is to determine a set of network parameters $\theta$. 
For any positive value of $p$, we define three components of the $L^p$ loss functions 
\begin{equation}\label{eq:Lp}
\begin{split}
    \mathcal{L}_r(\theta) &:= \iiint |\mathcal{L}_\mathrm{BGK}[f_\theta](t,{\bf x, v})|^p dtd{\bf x}d{\bf v}, \\
    \mathcal{L}_\mathrm{bc}(\theta) &:= \iiint |\mathcal{B}[f_\theta](t,{\bf x, v})|^p dtd{\bf x}d{\bf v}, \\
    \mathcal{L}_\mathrm{ic}(\theta) &:= \iint |f_\theta(0, {\bf x, v}) - f_0({\bf x,v})|^p d{\bf x}d{\bf v},
\end{split}    
\end{equation}
corresponding to the PDE, boundary condition, and initial condition, respectively. Subsequently, we construct a comprehensive PINN loss function
\begin{equation}\label{eq:pinn_loss}
    \mathcal{L}(\theta) = 
    \lambda_r \mathcal{L}_{r}(\theta)
    + \lambda_{bc}\mathcal{L}_{bc}(\theta)
    + \lambda_{ic} \mathcal{L}_{ic}(\theta).
\end{equation}
In this formulation, $\lambda_r$, $\lambda_\mathrm{bc}$, and $\lambda_\mathrm{ic}$ are positive real numbers, selected based on the empirical evaluation.
The integrals present in \eqref{eq:pinn_loss} are computed using appropriate numerical integration methods.
The optimal network parameters, denoted as $\theta^*$, are determined by solving
\[
    \theta^* = \arg \min_\theta \mathcal{L}(\theta),
\]
employing a suitable optimization algorithm, such as Adam \cite{kingma2014adam}.

\section{Methodology}\label{sec:methodology}
In this section, we address various challenges encountered in the application of PINNs to solve the BGK model.
We also propose a series of methodologies designed to effectively mitigate these difficulties.

\subsection{Methodology for reduction of computational cost}\label{sec:reducing_computational_cost}
Traditional mesh-based numerical solvers for the BGK model often face the challenge of dimensionality, particularly due to the incorporation of the microscopic velocity space.
In contrast, PINNs are not constrained by the need for structured meshes, which initially suggests an advantage in this context.
However, the requirement to compute integrals for macroscopic moments somewhat offsets this perceived benefit.
Furthermore, when dealing with solutions that have complex profiles, PINNs typically demand a large number of collocation points.
Therefore, efficiently solving the BGK model using PINNs demands a reduction in the number of network forward passes and a more rapid evaluation of these integrals.

\subsubsection{Separable Physics-informed Neural Networks}
The computation of macroscopic moments $\rho, {\bf u}, T$ in the context of PINNs significantly increases the number of network forward passes, especially when solving PDEs of similar dimensionality that do not involve such integrals.
Consider $m_{ijk}$ as the $ijk$-th moment, defined by the following expression: 
\begin{equation}\label{eq: moments}
    m_{ijk}(t, {\bf x}) := \int_{\mathbb{R}^3} f(t,{\bf x, v}) v_x^i v_y^j v_z^k d{\bf v}.
\end{equation}
Employing a trapezoidal rule with $N$ points allocated along each axis results in a total of $N^3$ points.
Consequently, for each pair of $(t, {\bf x})$, it becomes necessary to perform $O(N^3)$ evaluations of the function $f$ to accurately compute $m_{ijk}(t, {\bf x})$.

In this study, we utilize the Separable PINN framework \cite{cho2024separable} to effectively reduce the number of network forward passes required.
This approach is grounded in the principles of canonical polyadic (CP) decomposition, a technique known for its efficacy in managing high-dimensional problems \cite{hitchcock1927expression}; for other low-rank methods, see, e.g., \cite{guo2023local, GQ24}.
For instance, CP decomposition has been applied in classical numerical methods for solving the 3D BGK model \cite{boelens2020tensor}.
Consider $\phi$ as a scalar-valued function on $\mathbb{R}^d$.
Let $G$ denote a rectilinear grid in $\mathbb{R}^d$, characterized by $N_i$ points along the $i$-th axis.
It can be demonstrated that, given a sufficiently large value of $R$, there exist functions $\phi_1, \phi_2, \dots, \phi_d: \mathbb{R} \rightarrow \mathbb{R}^R$ such that for every point $(x_1, x_2, \dots, x_d)$ in the grid $G$, the following relationship holds: 
\[
    \phi(x_1, x_2, \dots, x_d) \approx \sum_{r=1}^R \prod_{i=1}^d \phi_{i,r}(x_i).
\]
In this expression, $\phi_{i,r}(x_i)$ denotes the $r$-th element of the vector $\phi_i(x_i)$.
Motivated by this, separable PINNs assign $d$ FNNs $\{ \phi(\cdot; \theta_i): \mathbb{R} \rightarrow \mathbb{R}^R \}_{i=1}^d$ to the $d$ axes of the domain, and then combine them in the same way to approximate $\phi$ at ${\bf x} = (x_1, \dots, x_d)\in \mathbb{R}^d$ by
\[
    \phi({\bf x}) \approx \phi_\theta^\mathrm{SPINN}({\bf x}) = \sum_{r=1}^R \prod_{i=1}^d \phi_r(x_i; \theta_i).
\]
It is important to highlight that for a rectilinear grid configured as $(N_1, N_2, \dots, N_d)$, separable PINNs significantly reduce the number of network forward passes from $O\left( \prod_{i=1}^d N_i \right)$ to $O\left(\sum_{i=1}^d N_i\right)$. This reduction offers a substantial computational advantage when compared to traditional FNNs.
Furthermore, the computation of partial derivatives can be efficiently executed using Jacobian-Vector Products, also known as forward-mode automatic differentiation.

Neural networks, when equipped with a sufficient number of hidden units, are recognized for their ability to approximate any continuous function.
However, the extension of this universal approximation capability to separable PINNs is not immediately obvious.
To address this, and to enhance the rigor of our paper, we provide the universal approximation theorem for separable PINNs (\Cref{thm:uat_spinn_f}).

\begin{theorem}\label{thm:uat_spinn}
Let $X$ and $Y$ be compact subsets of $\mathbb{R}^d$, and $\phi \in L^2 (X\times Y)$.
For $\epsilon >0$, there exists a separable PINN $\phi_\theta$ such that
\[
    \| \phi - \phi_\theta \|_{L^2(X\times Y)} < \epsilon.
\]
\end{theorem}
\begin{proof}
    See Supplementary Materials A. of \cite{cho2024separable}.
\end{proof}

\begin{remark}
    We note that the universal approximation theorem \cite{cybenko1989approximation} works for one hidden layer network.
    Its extension to an arbitrary, bounded depth network is easy, as outlined in \cite{kidger2020universal}:
    express the $l$-th layer as $B_l \circ \sigma \circ A_l$,
    where $A_l$ and $B_l$ are affine transformations and $\sigma$ is an activation function;
    let $l>1$-th layers approximate the identity function.
\end{remark}

\Cref{thm:uat_spinn} requires modifications to be suitably applied to the particle density function $f$ since it is defined over the microscopic velocity domain $\mathbb{R}^3$, rather than on a compact subset of $\mathbb{R}^3$. 
In light of these considerations, we present a tailored approximation theorem for $f$ as a function defined on the domain $\Omega \times \mathbb{R}^3$, ensuring its compatibility with the unique characteristics of the BGK model.

\begin{theorem}\label{thm:uat_spinn_f}
    For $d \in \{1, 2, 3\}$,
    let $\Omega \subset \mathbb{R}^d$ be a compact set and $f: \Omega \times \mathbb{R}^3 \rightarrow \mathbb{R}$ be a particle density function at a fixed time, and $\phi({\bf v}) = 1 + |{\bf v}|^2$.
    We assume that $f \in L^2(\Omega \times \mathbb{R}^3)$,
    and $\int_{\mathbb{R}^3} f\phi d{\bf v} \in L^2(\Omega)$.
    Then, for $\epsilon >0$, there exists a separable PINN $f_\theta$ such that
    \[
        \| f - f_\theta \|_{L^2(\Omega \times \mathbb{R}^3)} < \epsilon.
    \]
    Moreover, 
    \[
        \left\|\int_{\mathbb{R}^3} (f - f_\theta)\phi d{\bf v}\right\|_{L^2(\Omega)} < \epsilon.
    \]
\end{theorem}

\begin{proof}
    We first note that $f\phi$ exhibits rapid decay along the velocity domain since
    \[
        M:= \int_{\mathbb{R}^3} f \phi d{\bf v} \in L^2(\Omega).
    \]
    Let \(B(r)\) be a ball of the radius \(r\) centered at the origin.
    We define $M_r$ as $\int_{B(r)} f\phi dv$.
    Since $f\phi \ge 0$, we have $0 \le M_r \le M$, and $M_r \rightarrow M$ as $r \rightarrow \infty$.
    Therefore, with
    \[
        \int_\Omega (M-M_r)^2 dx \le \int_\Omega M^2 dx < \infty,
    \]
    the dominated convergence theorem implies $\|M_r - M \|_{L^2(\Omega)} \rightarrow 0$ as $r \rightarrow \infty$.
    Hence, we can choose a compact subset $A \subset \mathbb{R}^3$ such that
    \[
        \left \| \int_{A^c} f \phi d{\bf v} \right\|_{L^2(\Omega)} < \frac{\epsilon}{2}.
    \]
    On the other hand, from $f \in L^2(\Omega \times \mathbb{R}^3)$, we have
    \[
        \int_\Omega\int_{\mathbb{R}^3} f^2 dv dx < \infty.
    \]
    We choose a compact subset $B \subset \mathbb{R}^3$ such that
    \[
        \int_\Omega \int_{B^c} f^2  dv dx < \frac{\epsilon^2}{4}.
    \]
    Let $C = A\cup B$.
    We divide the microscopic velocity space $\mathbb{R}^3$ into two parts, $C$ and $C^c$.
    For $C$, by \cref{thm:uat_spinn}, there exists a separable neural network $f_\theta$ such that 
    \[\| f - f_\theta \|_{L^2(\Omega \times C)} < \frac{\epsilon}{2(1+|C|^2)\left(1 + m(C)^{1/2}\right)} \le \frac{\epsilon}{2},\]
    where \(|C|:= \sup\{|v|: v \in C\}\) and \(m(C) = \int_C dx\).
    For $C^c$, we extend $f_\theta$ by zero to $\Omega \times \mathbb{R}^3$.
    Then we have
    \[
    \begin{split}
    \| f - f_\theta\|_{L^2(\Omega \times \mathbb{R}^3)}
    & \le \| f - f_\theta\|_{L^2(\Omega \times C)} + \| f - f_\theta \|_{L^2(\Omega \times C^c)} \\
    & = \| f - f_\theta\|_{L^2(\Omega \times C)} + \| f \|_{L^2(\Omega \times C^c)} \\
    & \le \frac{\epsilon}{2(1 + |C|^2)} + \frac{\epsilon}{2} \le \epsilon.
    \end{split}
    \]
    Moreover,
    \[
    \begin{split}
        \left\| \int_{\mathbb{R}^3} (f - f_\theta) \phi d{\bf v} \right\|_{L^2(\Omega)}
        & \le \left \| \int_C (f-f_\theta)\phi d{\bf v}\right \|_{L^2(\Omega)}
        + \left \| \int_{C^c} (f-f_\theta)\phi d{\bf v}\right \|_{L^2(\Omega)} \\
        & = \left \| \int_C (f-f_\theta)\phi d{\bf v}\right \|_{L^2(\Omega)}
        + \left \| \int_{C^c} f\phi d{\bf v}\right \|_{L^2(\Omega)} \\
        & \le \left \| \int_C (f-f_\theta)\phi d{\bf v}\right \|_{L^2(\Omega)} + \frac{\epsilon}{2} \\
        & \le  \left[\int_\Omega \left(\int_C (f-f_\theta)^2d{\bf v}\int_C \phi^2 d{\bf v}\right) dx\right]^{1/2} + \frac{\epsilon}{2} \\ 
        & \le \| f-f_\theta\|_{L^2(\Omega\times C)}(1+|C|^2)^{1/2}m(C)^{1/2} + \frac{\epsilon}{2}
        \le \epsilon.
    \end{split}
    \]
\end{proof}
\begin{remark}
Note that our convergence proof is carried out in $L^2$ space whereas the kinetic distribution functions lie in $L^1$ space. The main reason is that the underlying Theorem \ref{thm:uat_spinn} is established in $L^2$ space using the orthogonality in $L^2$. 
We, however, mention that such $L^2$ based convergence analysis provides a relevant perspective for the convergence of our simulation since we are working on a truncated domain so that we have by H{\"o}lder's inequality
\[
    \|f-f_{\theta}\|_{L^1} \leq C_\gamma \|f-f_\theta\|_{L^2},
\]
where $\gamma$ is the truncation parameter.
\end{remark}

\subsubsection{Fast Evaluation of the Macroscopic Moments}
Consider $N_v$ as the number of points on each axis of the microscopic velocity space, utilized for evaluating the macroscopic moments as outlined in \eqref{eq: moments}.
As previously discussed, separable PINNs necessitate $O(3N_v)$ network forward passes for the computation of $N_v^3$ function values. 
However, directly evaluating the integrals from the output of the separable PINN is not computationally efficient, as it requires $O(N_v^3)$ operations and memory cost.
To optimize this process, we leverage the linear nature of moment calculation in \eqref{eq: moments} and the inherent separable structure of the separable PINNs. 
Afterwards, the macroscopic moments can be evaluated more efficiently, within $O(3 R N_v)$ operations.

The integrals for the macroscopic moments \(m_{ijk}\) can be computed as follows:
\begin{equation}\label{eq:separable_integral}
    \begin{split}
            \int f_\theta(t,{\bf x,v}) v_x^{i_x} v_y^{i_y} v_z^{i_z} d{\bf v}
            &= \int \sum_{r=1}^R f_r(t;\theta_t) \prod_{p \in \{x,y,z\}} f_r(p;\theta_p) f_r(v_p; \theta_{v_p}) v_p^{i_p} d{\bf v} \\
            &= \sum_{r=1}^R f_r(t; \theta_t) \prod_{p \in \{x,y,z\}}  f_r(p; \theta_p) \int  f_r(v_p; \theta_{v_p})v_p^{i_p} dv_p.
    \end{split}
\end{equation}
Consequently, \Cref{eq:separable_integral} significantly reduces both the computational cost and memory requirements to $O(3RN_v)$, a substantial improvement over the $O(N_v^3)$ required for standard computation.
{\Cref{tab:integral_speed} illustrates elapsed times for calculating $(\rho, {\bf u}, T)$ for SPINN with two distinct integration strategies as a function of the rank, $R$, and the number of quadrature points, $N$, per velocity axis.
We considered only one spatial dimension ($d=1$),
and discretized the space-time domain by $(N_t, N_x) = (12,12)$, respectively.
Here, SPINN has five FNNs which have a width of $128$, and a depth of $3$ for each axis.
As a baseline, single FNN-based methods such as vanilla PINN, resulted in out-of-memory when $N=65$ even though we employed only one hidden layer.
For $N=33$, it took $6.4m \pm 621n$ seconds, which is less efficient than SPINN.
For both the ``Standard'' and ``Separate'' integration strategies for SPINN, there appears to be no clear relationship between $R$ and computing times.
This lack of correlation can be attributed to the following factors: In the Standard strategy, a significant portion of the computing time is consumed by the integration phase, with its duration depending solely on $N$.
Conversely, in the Separate integration strategy, the computational cost of the network's forward pass dominates over integration, once again dependent solely on $N$.
However, the computing times for Standard integration increase approximately $6 \sim 8$ times as $N$ doubles.
For $N=257$, we observed the out-of-memory issue for standard integration immediately after increasing $R=1$ to $R=2$.
In contrast, Separate integration shows excellent scalability, as the computing times remain similar across all configurations.
Although SPINN achieves notable efficiency gains through its network architectural design compared to the basic form like FNNs, the memory constraints and limited scalability of standard integration become significantly more pronounced as the spatial dimension increases.
Therefore, this efficiency of the separate integration strategy enables the fast evaluation of integrals on fine quadrature points, assuming that the rank $R$ is not very large.
}
In this research, we have chosen to use $N_v = 257$ and $R \in \{128, 256\}$ for our computations.

\begin{table}
    \centering
    \begin{tabular}{|c|c|c|c|c|}
    \hline
        (R, N) & (32,65) & (32,129) & (64,65) & (64,129) \\ \hline
        Standard & $892\mu \pm 3.57\mu$ & $5.57m \pm 3.53\mu$ & $899\mu \pm 882n$ & $5.78m \pm 7.27\mu$ \\
        Separate & $167\mu \pm 1.78\mu$ & $167\mu \pm 336n$ & $169\mu \pm 1.66\mu$ & $169\mu \pm 107n$ \\
    \hline \hline
        (R, N) & (128,65) & (128,129) & (1,257) & (1024,1025)\\ \hline
        Standard & $988\mu \pm 935n$ & $6.02m \pm 4.42\mu$ & $32.2m \pm 33.7\mu$ & OOM\\
        Separate & $166\mu \pm 1.39\mu$ & $168\mu \pm 394n$ & $150\mu \pm 6.69\mu$ & $213\mu \pm 3.4\mu$\\
    \hline
    \end{tabular}
    \caption{
    Computation times (mean $\pm$ std) for calculating $(\rho, {\bf u}, T)$ for SPINN.
    ``Separate'' corresponds to the integration strategy outlined in \Cref{eq:separable_integral}.
    ``Standard'' corresponds to the strategy without \Cref{eq:separable_integral}.
    Here, $(n, \mu, m) = (10^{-9}, 10^{-6}, 10^{-3})$ seconds, respectively.
    $R$ is the rank of the SPINN.
    $N$ is the number of quadrature points per velocity axis.
    We considered $d=1$ case, with $(N_t, N_x) = (12, 12)$.
    OOM means ``Out Of Memory''.
    We only considered $R=1$ for $N=257$, as $R=2$ produces OOM for Standard integration.
    All measurements were conducted on single precision, via \texttt{\%timeit} command of iPython.
    }
    \label{tab:integral_speed}
\end{table}

\subsection{Strategies for Enhancing Accuracy}
Neural networks are intricately designed and trained to precisely approximate macroscopic moments $\rho, {\bf u}, T$. However, there are scenarios where, despite the particle density $f$ being accurately approximated to a certain level of mean squared error, the macroscopic moments may still be inaccurately represented.
This issue often stems from the accumulation of numerical errors during the integration process, as the macroscopic moments are derived by integrating the particle density.
To mitigate this challenge, we introduce a neural network architecture tailored for kinetic equations, complemented by a relative loss function, aimed at enhancing the accuracy of these macroscopic moments.

\subsubsection{Maxwellian Splitting}
The BGK operator characterizes the evolution of an initial particle density function towards a state of thermodynamic equilibrium.
Consequently, it is natural to conceptualize the particle density function as a composite of two components: the local equilibrium part and the residual part.
This decomposition approach aligns with the methodologies proposed in the works of Lou et al. \cite{lou2021physics} and Li et al. \cite{li2023solving}, where the motivation might be brought from \cite{bennoune2008uniformly,jin2010micro}.
Inspired by the decomposition approach, we propose the employment of two distinct separable PINNs: $f_\theta^\mathrm{eq}$ and $f_\theta^\mathrm{neq}$.
The first network, $f_\theta^\mathrm{eq}$, is designed to accept inputs $t, {\bf x}$ and outputs a tuple $(\rho_\theta, {\bf u}_\theta, T_\theta) \in \mathbb{R}^{1+3+1}$. 
The second network, $f_\theta^\mathrm{neq}$, processes the full set of variables $t, {\bf x, v}$, yielding a real-valued output. These networks are combined as follows:
\begin{equation}\label{eq:maxwellian_splitting}
    f_\theta(t,{\bf x, v}) = M[\rho_\theta(t, {\bf x}), {\bf u}_\theta(t, {\bf x}), T_\theta(t, {\bf x})]({\bf v}) + \alpha f_\theta^\mathrm{neq}(t, {\bf x, v})
\end{equation}
to approximate the solution to the BGK model.
Here, $M[\rho, {\bf u}, T]$ denotes the Maxwellian distribution parameterized by the macroscopic moments $\rho, {\bf u}, T$, albeit used somewhat loosely in this notation.
The coefficient $\alpha$ is a parameter dependent on the specific problem.
While one might consider setting $\alpha$ equal to the Knudsen number (Kn) following the Chapman-Enskog expansion \cite{chapman1990mathematical}, in this study, we have fixed $\alpha$ at \(1\) since it performed well throughout all the experiments.

To compute the macroscopic moments as specified in \eqref{eq: moments}, we perform separate integrations for the equilibrium term $M[\rho_\theta, {\bf u}_\theta, T_\theta]$ and the non-equilibrium term $f_\theta^\mathrm{neq}$, taking advantage of the linear nature of the moment calculation.
Importantly, for the equilibrium term, the zeroth, first, and second moments are directly given by $\rho_\theta$, $\rho_\theta {\bf u}_\theta$, and $\rho_\theta (3 T_\theta + |{\bf u}_\theta|^2)$, respectively, eliminating the need for numerical integration.
The non-equilibrium term, in contrast, is efficiently integrated using the method described in \eqref{eq:separable_integral}.
Although this decomposition approach does not have formal theoretical justification, empirical evidence indicates that splitting the density function into these two terms significantly improves both training speed and accuracy.

\subsubsection{Integration of Decay Property into Neural Networks}\label{sec:multiplying_gaussian_functions}
Accurate computation of macroscopic moments necessitates a more careful consideration of high relative velocities. Given that the density function typically exhibits rapid decay as the magnitude of ${\bf v}$ increases, contributions from high relative velocities to the macroscopic moments are generally negligible.
However, if the neural network employed does not mirror this rapid decay, its output at high relative velocities might contribute excessively, leading to inaccuracies in the calculated macroscopic moments.
Therefore, ensuring that the neural network appropriately reflects the decay characteristics of the particle density function is crucial for the precision of the macroscopic moment calculations.

For an illustration, we consider the standard Gaussian function $g$ on the truncated domain $\mathcal{D} := (-10, 10) \subset \mathbb{R}$:
\[
    g: v \in \mathcal{D} \mapsto \frac{1}{\sqrt{2\pi}}e^{-\frac{v^2}{2}}.
\]
When performing numerical integration of its moments, the function values near the boundary of the domain $\partial \mathcal{D}$ (specifically around $\pm 10$) are so small that the contribution of $v$ near \(\partial \mathcal{D}\) to the moments should be negligible.
However, during the training phase, a neural network $g_\theta$ that approximates $g$ may not exhibit the same rapid decay as $g$.
To investigate this issue, we employ an FNN, denoted as $g_\theta$, to approximate the Gaussian function $g$.
This approximation is achieved by minimizing a loss function $\mathcal{L}_\mathrm{MSE}$, which is defined as follows:
\begin{equation}\label{eq:l2}
    \mathcal{L}_\mathrm{MSE}: \theta \mapsto \frac{1}{20}\int_{-10}^{10} |g_\theta(v) - g(v)|^2 dv.
\end{equation}
The loss function was evaluated with the Monte-Carlo integration.
Additional details can be found in \Cref{tab:toy_experiments}.

\begin{table}
\centering
\begin{tabular}{ |c|c|c|c| } 
 \hline
    layers & (1, 256, 1) \\
    activation function & $\tanh$ \\
    training points & 16, uniform distribution \\
    optimizer & Adam \\
    the number of iterations & 1M \\
    initial learning rate & $10^{-4}$ \\
    learning rate schedule & cosine decay to $0$ \cite{loshchilov2016sgdr} \\
    initial $(\mu, \tau)$ & $(2, 1/\sqrt{2})$ \\
 \hline
\end{tabular}
\caption{
    Additional details for toy experiments in \Cref{sec:multiplying_gaussian_functions,sec:relative_loss}.
}
\label{tab:toy_experiments}
\end{table}

\begin{figure}[htbp]
	\centering
    \includegraphics[width=0.9\textwidth]{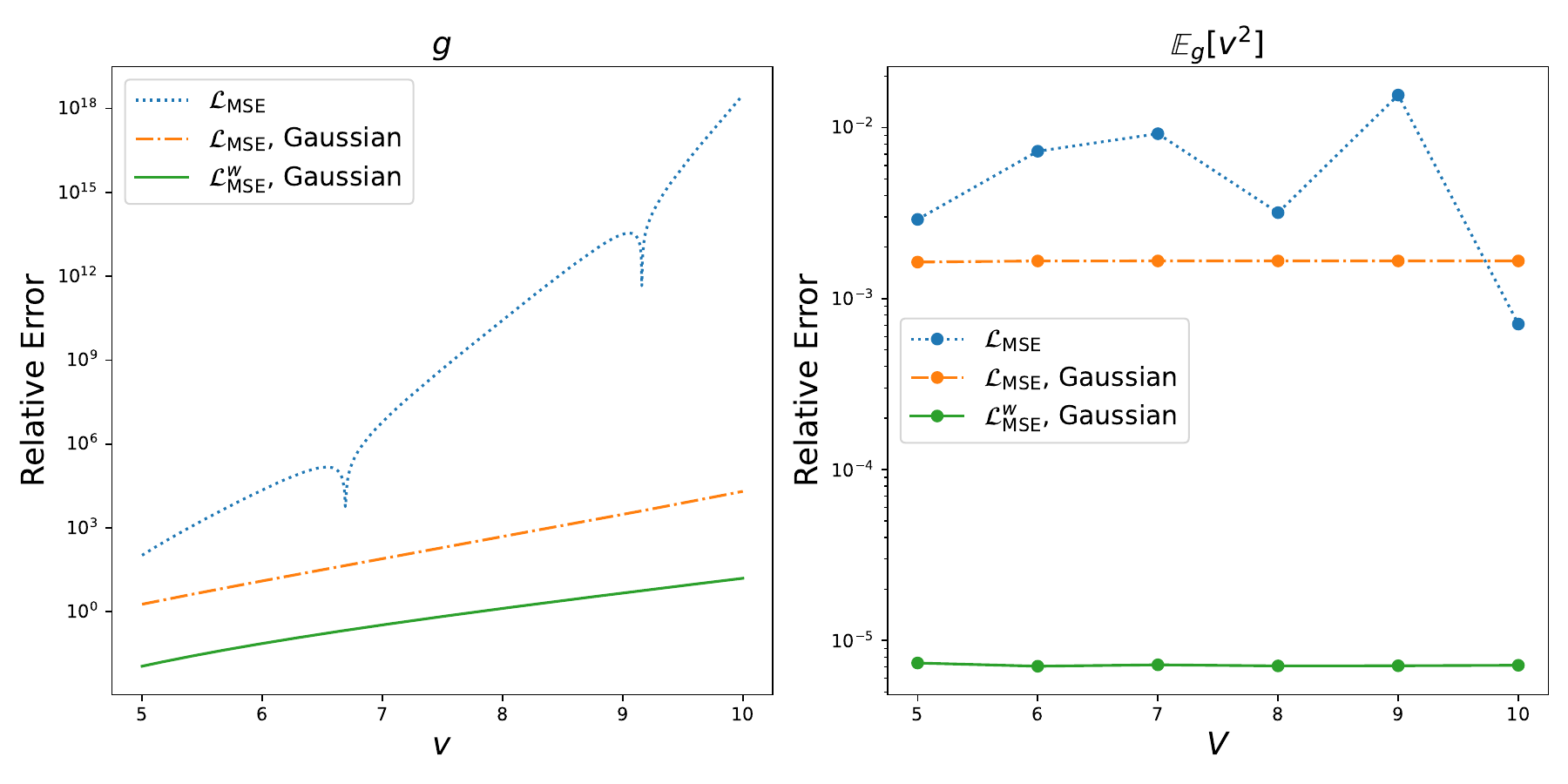}
    \caption{
    Left: \(v \in (5, 10) \mapsto |g_\theta(v) - g(v) / |g(v)|\).
    Right: $V\mapsto \int_{-V}^V v^2 |g_\theta(v) - g(v)|dv / \int_{-V}^V v^2g(v)dv$.
    Uniform meshes of $1024$ points are used to evaluate pointwise errors and integrals.
    Dotted line: results for the \(\mathcal{L}_\mathrm{MSE}\) loss \eqref{eq:l2} approximation with a neural network (baseline).
    Dot dashed and solid lines: results for \(\mathcal{L}_\mathrm{MSE}\) loss and relative loss \(\mathcal{L}_\mathrm{MSE}^w\) \eqref{eq:relative_loss}, respectively, with a neural network multiplied by $g_p$ \eqref{eq:gaussian}.
    }
	\label{fig:gaussian_regression}
\end{figure}

\Cref{fig:gaussian_regression} shows the outcomes of three distinct numerical experiments designed to assess various aspects of our methodology. These experiments include an investigation into the effects of the approximated tail values of $g_\theta$ on the second moment, an examination of the impact of incorporating a Gaussian function into the neural network (referred to as Gaussian in the figure legends), which is aimed at facilitating the decay speed, and an effect of employing a relative loss function \(\mathcal{L}_\mathrm{MSE}^w\), denoted as (to be defined in \eqref{eq:relative_loss} below), that imposes large weights on small-scale features.
The left panel of \Cref{fig:gaussian_regression} presents relative point-wise error $|g_\theta(v) - g(v)| / |g(v)|$ on a part of the domain $(5, 10)$.
The right panel presents errors for the second moment evaluated on the truncated interval $(-V, V)$.
The dotted lines correspond to the results of minimizing \(\mathcal{L}_\mathrm{MSE}\) function, without any modifications in network architecture.
As observed, the relative pointwise error exhibits an exponential increase with the growth of $v$. This pattern suggests that the optimized neural network, $g_\theta$, does not decay as rapidly as $g$. A contributing factor to this behavior is the nature of the \(\mathcal{L}_\mathrm{MSE}\) loss function, which tends to relatively overlook the smaller-scale values.
Consequently, the errors observed in the second moment are not uniform, implying that $g_\theta$ possesses thicker tails in comparison to $g$. Furthermore, despite being a regression setting, the magnitude of the errors is notably significant, particularly above $10^{-2}$ when $V = 9$. This level of error could lead to substantial inaccuracies in training PINNs, given that direct access to the exact solution of the PDE is typically unavailable.

The equilibrium term $M[\rho_\theta, {\bf u}_\theta, T_\theta]$ in \eqref{eq:maxwellian_splitting} decays as fast as $g$,
yet the non-equilibrium term $f_\theta^\mathrm{neq}$ does not.
To make $f_\theta^\mathrm{neq}$ decay exponentially fast to $|{\bf v}|$, we modified its form slightly into 
\[
    f_\theta^\mathrm{neq}(t,{\bf x, v})
    = \sum_{r=1}^R  f_r(t;\theta_t) \prod_{p \in \{x,y,z\}} f_r(p;\theta_p) \left[ f_r(v_p;\theta_p) g_p(v_p; \tau, \mu_p) \right],
\]
where
\begin{equation}\label{eq:gaussian}
    g_p: v_p \mapsto e^{- \tau^2 (v_p - \mu_p)^2 / 2}.
\end{equation}
We set $\tau, \mu:= (\mu_x, \mu_y, \mu_z)$ as network parameters to avoid hand-tuning procedures.
We approximated the function $g$ by a neural network multiplied by $g_p$.
Initial values for parameters $\tau$, and $\mu$ were set to $1/\sqrt{2}$, and $2$, respectively.
Note that the standard Gaussian function corresponds to $\tau = 1$ and $\mu = 0$.
The dot-dashed lines of \Cref{fig:gaussian_regression} present the result.
From the left panel, we can observe that values for large $|v|$ are well captured, compared to the dotted line (baseline).
As we can see from the right panel, even though the approximated second moment was better than the baseline overall, they were in the same order of magnitude.

\subsubsection{Relative Loss}\label{sec:relative_loss}
As previously discussed, the standard $L^2$ loss function \(\mathcal{L}_\mathrm{MSE}\) tends to prioritize learning large-scale features.
However, in our context, small-scale features are equally crucial, particularly since our focus is on accurately approximating the macroscopic moments \(\rho, \mathbf{u}, T\).
Therefore, to appropriately emphasize different scales, we propose the adoption of a weighted \(L^2\) loss function, defined as:
\begin{equation}\label{eq:relative_loss}
    \mathcal{L}_\mathrm{MSE}^w: \theta \mapsto \int (w \mathcal{N}[f_\theta] )^2 dtd{\bf x}d{\bf v},
\end{equation}
where $\mathcal{N}$ represents either the PDE residual operator, the boundary condition, or the initial condition.
The weighting function $w$ is formulated as $1 / (|f_\theta| + \epsilon)$, stopping gradient computation in $\theta$.
The term $\epsilon$ is included to ensure numerical stability, with a practical value of $10^{-3}$ typically employed. For the initial condition loss, $w[f_\theta]$ can be substituted with $w[f_0]$.
It is important to note that $\mathcal{L}_\mathrm{MSE}^w$ assigns larger weights when the magnitude of $f_\theta$ is small, and smaller weights when it is large, thus enabling more balanced learning of different scale features.

We conducted a demonstrative example, as illustrated in \Cref{fig:gaussian_regression}, where we approximated the function $g$ using a neural network in conjunction with \eqref{eq:gaussian}, employing the loss function defined in \eqref{eq:relative_loss}.
The outcomes of this approximation are represented by the solid lines in \Cref{fig:gaussian_regression}.
The results clearly indicate a significant improvement over the dot-dashed line, which represents the approximation using the standard $L^2$ loss function \(\mathcal{L}_\mathrm{MSE}\) with uniform weights, as per \eqref{eq:l2}.
Notably, the relative error for the second moment was reduced by two orders of magnitude when using our proposed method.

\section{Numerical Results}\label{sec:numerical_results}
In this section, we present a series of numerical experiments to demonstrate the efficacy of SPINN-BGK.
The simulations cover a range of problems: two in 1D, two in 2D, and a particularly challenging one in 3D.
For the 1D and 2D problems, we generated reference solutions employing a high-order conservative semi-Lagrangian scheme \cite{cho2021conservative2} on a personal computer equipped with an AMD Ryzen 9 5900X 12-Core Processor operating at 3.70 GHz and 64GB of RAM, using Matlab \cite{matlab}.
However, due to the high computational costs associated with the curse of dimensionality, generating reference solutions for 3D problems was not feasible, as suitable numerical schemes are not readily available and there is a lack of open-source implementations for such complex cases.
This is precisely where our research makes a substantial contribution.
The SPINN-BGK successfully computes a numerical solution for the 3D case efficiently, even on a single GPU, showcasing its potential to address the computational barriers typically encountered in high-dimensional simulations.

Throughout these experiments, we consider three different Knudsen numbers, $\mathrm{Kn} \in \{10^0, 10^{-1}, 10^{-2}\}$.
The activation function for our MLP is set as $x \mapsto \sin(w_0 x)$, with network parameters initialized following the guidelines in \cite{sitzmann2020implicit}.
Empirically, we chose $w_0 = 10$ for our experiments.
For the training of SPINN-BGK, we utilized the Lion optimizer \cite{chen2023symbolic}, which bases its updates on the sign of the gradient momentum rather than the gradient momentum itself.
We observed that the Lion optimizer tends to achieve a lower training loss compared to the Adam optimizer \cite{kingma2014adam}.

To evaluate the accuracy of SPINN-BGK solutions, we employed the relative $L^2$ error metric, defined as $\| \phi_\mathrm{true} - \phi_\mathrm{pred} \|_2 / \|\phi_\mathrm{true}\|_2$ for $\rho$ and $T$, and $\| \phi_\mathrm{true} - \phi_\mathrm{pred} \|_2 / (\|\phi_\mathrm{true}\|_2 + 1)$ for $\bf u$, to avoid division by zero.
All experiments were conducted on a single NVIDIA RTX 4090 GPU, implemented using JAX \cite{jax2018github}.
Further experimental details are provided in \eqref{tab: detail}.
A comparison of the computing time between our SPINN-BGK approach and the conservative semi-Lagrangian scheme, as described in \cite{cho2021conservative2}, is provided in \ref{app:a}.

\begin{table}
\centering
\begin{tabular}{ |c|c|c|c| } 
 \hline
    width & 128 \\
    $R$ & 128 (smooth), 256 (Riemann) \\
    depth & 3 \\
    optimizer & Lion \\
    the number of iterations & 100K \\
    initial learning rate & $10^{-5}$ \\
    learning rate schedule & cosine decay to $0$ \cite{loshchilov2016sgdr} \\
    $(\lambda_r, \lambda_\mathrm{ic}, \lambda_\mathrm{bc})$ & $(1, 10^3, 1)$ \\
    initial $(\mu, \tau)$ & $({\bf 0}, 1)$ \\
 \hline
\end{tabular}
\caption{
    Additional details for \Cref{sec:numerical_results}.
    We conducted all experiments in \Cref{sec:numerical_results} with this setting for consistency.
}
\label{tab: detail}
\end{table}

\subsection{1D Smooth Problem}\label{sec:1d_smooth}
In this section, we conduct a test on the one-dimensional smooth problem as outlined in section 6.1 of \cite{li2023solving}.
The temporal scope of our test is confined to the interval $(0, 0.1]$.
We consider a spatial domain of $(-0.5, 0.5)$, imposing periodic boundary conditions.
For the microscopic velocity space, the computational domain is set to $(-10, 10)^3$.
The initial condition for this test is defined by a Maxwellian distribution characterized by specific macroscopic moments
\[
    \rho_0(x) = 1 + 0.5\sin(2\pi x), \quad
    {\bf u}_0(x) = {\bf 0}, \quad 
    T_0(x) = 1 + 0.5\sin(2\pi x + 0.2).
\]
In this experiment, the trapezoidal rule with $257$ points was employed for the numerical evaluation of macroscopic moments.
The optimization process involved $100K$ gradient descent steps. For each descent step, collocation points were sampled in the size of $(N_t, N_x, N_v^3) = (12, 16, 12^3)$.

\begin{figure}[htbp]
    \centering
    \includegraphics[width=\textwidth]{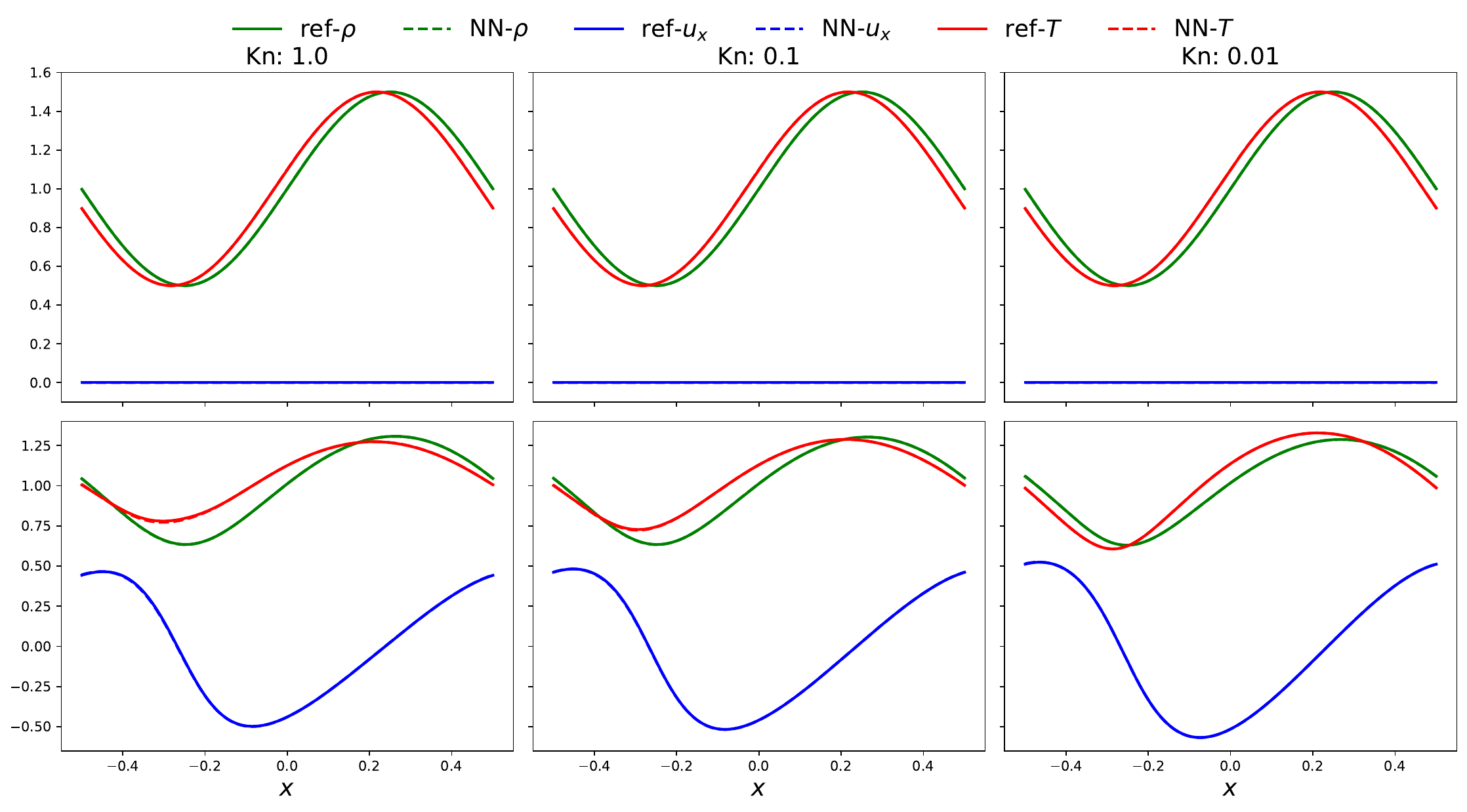}
    \caption{
        Numerical solutions to \cref{sec:1d_smooth}.
        Macroscopic moments $\rho, u_x, T$ are presented.
        The top row corresponds to $t=0$.
        The bottom row corresponds to $t=0.1$.
        Left column: $\mathrm{Kn}$ = 1.0.
        Middle column: $\mathrm{Kn}$ = 0.1.
        Right column: $\mathrm{Kn}$ = 0.01.
        The solid lines are the reference solutions, and the dotted lines are the neural network predictions.
    }
    \label{fig:1d_smooth}
\end{figure}

The numerical results are depicted in \Cref{fig:1d_smooth}.
The entire computation was completed in approximately 4 minutes.
For comparison, the reference solution was generated using a finer grid with $N_x = 1280$ and $N_v^3 = 25^3$.
It is important to note that $u_y$ and $u_z$ are omitted in our analysis, as they do not exhibit significant dynamics along the $y$ and $z$ directions.
The plot demonstrates that the predicted macroscopic moments align closely with the reference solutions for a range of Knudsen numbers, specifically $\mathrm{Kn} \in \{10^0, 10^{-1}, 10^{-2}\}$.
In \Cref{tab:1d_smooth}, we present the relative $L^2$ errors between the neural network predictions and the reference solutions.
Notably, all errors are equal to or less than the order of $O(10^{-3})$, which signifies a strong agreement between SPINN-BGK predictions and the reference solutions.
This result is indicative not only of a qualitative match but also of a quantitative alignment, underscoring the effectiveness of SPINN-BGK in accurately capturing the dynamics of the system.

\begin{table}
\centering
\begin{tabular}{ |c|c|c|c| } 
 \hline
  $(t=0,t=0.1)$ & $\mathrm{Kn} = 1$ & $\mathrm{Kn} = 0.1$ & $\mathrm{Kn} = 0.01$ \\
  \hline
  $\rho$ & (2.74e-4, 4.65e-4) & (1.81e-4, 3.99e-4) & (1.92e-4, 3.47e-4) \\
  $u_x$ & (4.19e-5, 1.36e-3) & (2.63e-5, 8.44e-4) & (3.66e-5, 3.53e-4) \\
  $T$ & (2.16e-4, 2.89e-3) & (1.49e-4, 1.58e-3) & (1.95e-4, 2.07e-4) \\
 \hline
\end{tabular}
\caption{
    Relative $L^2$ errors for the macroscopic moments $\rho, u_x, T$ of the 1D smooth problem.
    For each parenthesis, the first component corresponds to $t=0$ and the second component corresponds to $t=0.1$.
    }
\label{tab:1d_smooth}
\end{table}

\subsection{1D Riemann Problem}\label{sec:1d_riemann}
In this section, we conduct a test on the one-dimensional Sod tube problem, as detailed in section 6.2 of \cite{li2023solving}.
The computational settings for this test largely mirror those used in \Cref{sec:1d_smooth}, with the exceptions being the boundary condition, the number of collocation points, and the initial condition.
For the spatial domain, we implement the free-flow boundary condition.
The initial condition is represented by a local Maxwellian distribution with macroscopic moments $(\rho_0, {\bf u}_0, T_0)$.
Specifically, the values on the interval $(-0.5, 0)$ are set to $(1, {\bf 0}, 1)$, and on $[0, 0.5)$, they are $(0.125, {\bf 0}, 0.8)$.
Given the challenge of approximating jump functions with neural networks, we opt for a smoothed version of the macroscopic moments:
\[
    \rho_0(x) = 1 - 0.875 H(x), \quad
    {\bf u}_0(x) = {\bf 0}, \quad
    T_0(x) = 1 - 0.2 H(x).
\]
Here, $H: x \mapsto \left(1 + \tanh(100x)\right)/2$ serves as a smoothed approximation of the Heaviside function.
In each iteration of the process, we randomly sample points in the configuration $(N_t, N_x, N_{v_x}, N_{v_y}, N_{v_z}) = (12, 32, 32, 12, 12)$.

\begin{figure}[htbp]
    \centering
    \includegraphics[width=\textwidth]{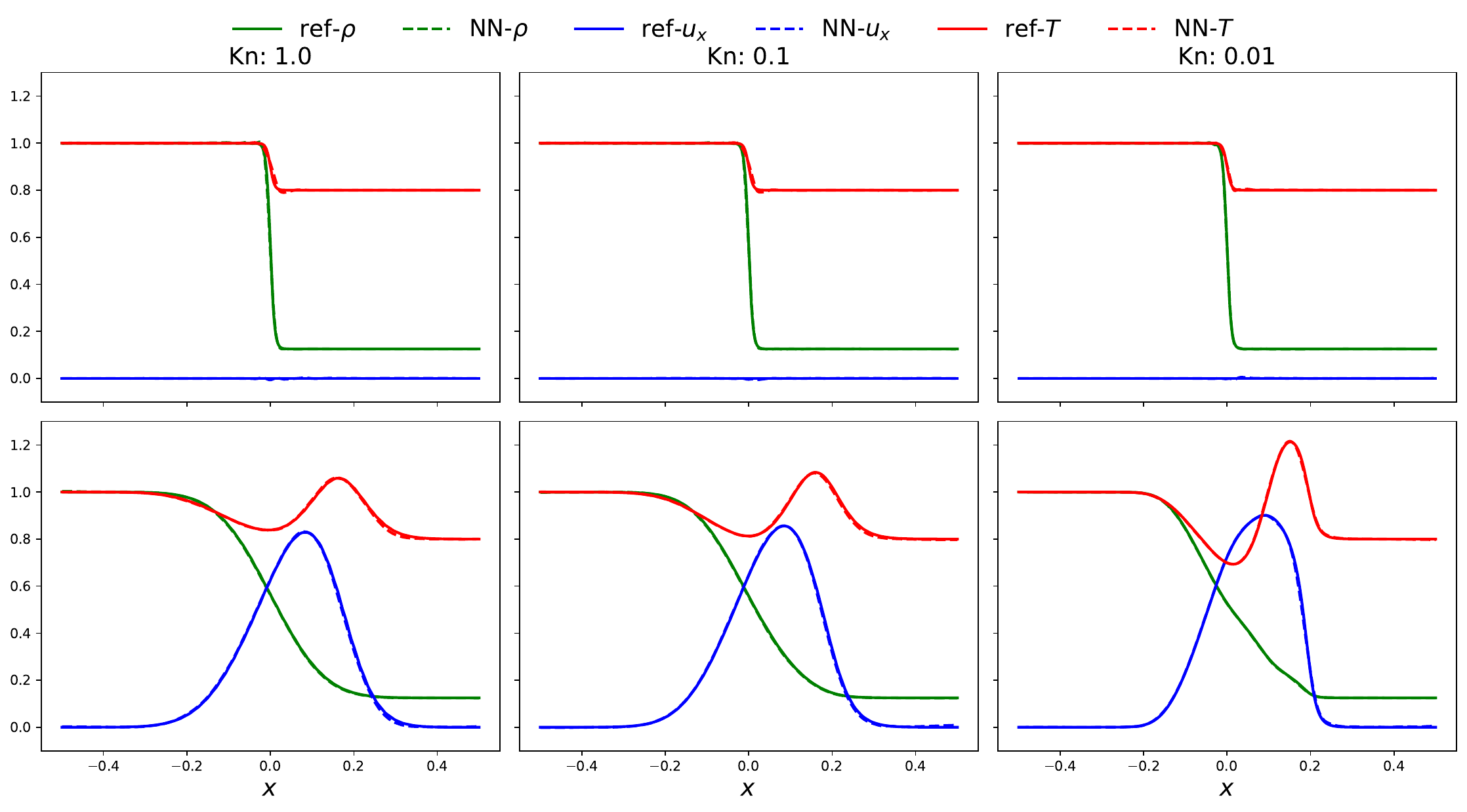}
    \caption{
        Predicted macroscopic moments $\rho, u_x, T$ of the 1D Riemann problem.
        Top row: predicted initial data.
        Bottom row: predicted macroscopic moments at $t=0.1$.
        Left column: $\mathrm{Kn} = 1.0$.
        Middle column: $\mathrm{Kn} = 0.1$.
        Right column: $\mathrm{Kn} = 0.01$.
    }
    \label{fig:1d_riemann}
\end{figure}

In \Cref{fig:1d_riemann}, we present the macroscopic moments obtained through SPINN-BGK alongside the reference solution.
The entire computation was completed in approximately 4 minutes and 10 seconds.
The reference solution was generated using a grid with $N_x = 1280$ and $(N_{v_x}, N_{v_y}, N_{v_z}) = (97, 25, 25)$. The figure demonstrates that the neural network solutions closely align with the reference solutions.
Relative $L^2$ errors for this experiment are detailed in \Cref{tab:1d_riemann}.
Notably, all errors are equal to or less than the order of $O(10^{-3})$, which quantitatively demonstrates a strong alignment between SPINN-BGK solutions and reference solutions.
This level of accuracy underscores the effectiveness of our method in capturing the essential dynamics of the problem.

\begin{table}
\centering
\begin{tabular}{ |c|c|c|c| } 
 \hline
  $(t=0, t=0.1)$ & $\mathrm{Kn} = 1$ & $\mathrm{Kn} = 0.1$ & $\mathrm{Kn} = 0.01$ \\
  \hline
  $\rho$ & (6.72e-3, 2.96e-3) & (5.45e-3, 3.08e-3) & (2.76e-3, 1.77e-3) \\
  $u_x$ &(1.00e-3, 4.65e-3) & (9.50e-4, 4.10e-3) & (8.47e-4, 3.67e-3) \\
  $T$ & (4.58e-3, 4.92e-3) & (3.83e-3, 4.58e-3) & (1.96e-3, 3.02e-3) \\
 \hline
\end{tabular}
\caption{Relative $L^2$ errors for the macroscopic moments $\rho, u_x, T$ of the 1D Riemann problem.}
\label{tab:1d_riemann}
\end{table}

\subsection{2D Smooth Problem}\label{sec:2d_smooth}
In this section, we conduct a test on the two-dimensional smooth problem as outlined in Section 6.3 of \cite{li2023solving}.
The experimental setup largely follows that of \Cref{sec:1d_riemann}, with a few notable adjustments.
The spatial domain is set as $(-0.5, 0.5)^2$, over which we impose periodic boundary conditions.
The initial condition is defined by a Maxwellian distribution with specific macroscopic moments 
\[
    \rho_0(x,y) = 1 + 0.5 \sin(2\pi x) \sin(2 \pi y), \quad
    {\bf u}_0(x,y) = {\bf 0}, \quad
    T_0(x,y) = 1.
\]
For this test, we utilize \(12\) collocation points per axis.
The relative $L^2$ errors obtained from our experiment are detailed in \Cref{tab:2d_smooth}.
Remarkably, the entire computation, including both training and inference phases, is completed in approximately 5 minutes and 10 seconds.
The reference solution for comparison was generated using a grid with $N_x^2 = 160^2$ and $N_v^3 = 25^3$.
Our results demonstrate that we can achieve solutions that are quantitatively well-aligned with the reference solutions in the two-dimensional case.

\begin{table}
\centering
\begin{tabular}{ |c|c|c|c| } 
 \hline
  $(t=0, t=0.1)$ & $\mathrm{Kn} = 1$ & $\mathrm{Kn} = 0.1$ & $\mathrm{Kn} = 0.01$ \\
  \hline
  $\rho$ & (1.82e-4, 1.29e-4) & (1.71e-4, 1.16e-4) & (1.87e-4, 1.53e-4) \\
  $u_x$ & (1.21e-5, 1.17e-4) & (1.62e-5, 1.08e-4) & (1.40e-5, 1.12e-4) \\
  $u_y$ & (1.37e-5, 1.07e-4) & (9.46e-6, 8.75e-5) & (1.06e-5, 1.23e-4) \\
  $T$ & (1.50e-5, 4.14e-5) & (1.85e-5, 3.50e-5) & (1.61e-5, 8.51e-5) \\
 \hline
\end{tabular}
\caption{Relative $L^2$ errors for the macroscopic moments $\rho, u_x, u_y, T$ of the 2D smooth problem.}
\label{tab:2d_smooth}
\end{table}

\subsection{2D Riemann Problem} \label{sec:2d_riemann}
In this section, we conduct a test on the two-dimensional Riemann problem as presented in \cite{cho2021conservative2}.
The spatial domain for this test is defined as $(-1, 1)^2$, where we apply a free-flow boundary condition.
Within this domain, we consider a circle $S$ defined by $x^2 + y^2 = 0.2$.
The initial condition is a local Maxwellian distribution with macroscopic moments $(\rho_0, {\bf u}_0, T_0)$.
Inside the sphere, the values are set to $(1, {\bf 0}, 1)$, while outside the sphere, they are $(0.125, {\bf 0}, 0.8)$.
To address the discontinuity in the initial data, similar to the approach in the 1D Riemann problem, we employ a relaxed version using the Heaviside function $H(r^2) =  \left(1+\tanh(100r^2)\right)/2$, where $r^2 = 0.2 - x^2 - y^2$.
The computational domain for the microscopic velocity space is established as $(-6, 6)^3$. For each iteration of the process, collocation points are sampled in the configuration $(N_t, N_x^2, N_v^3) = (12, 12^2, 12^3)$.

\begin{figure}
     \centering
     \begin{subfigure}[b]{0.69\textwidth}
         \centering
         \includegraphics[width=\textwidth]{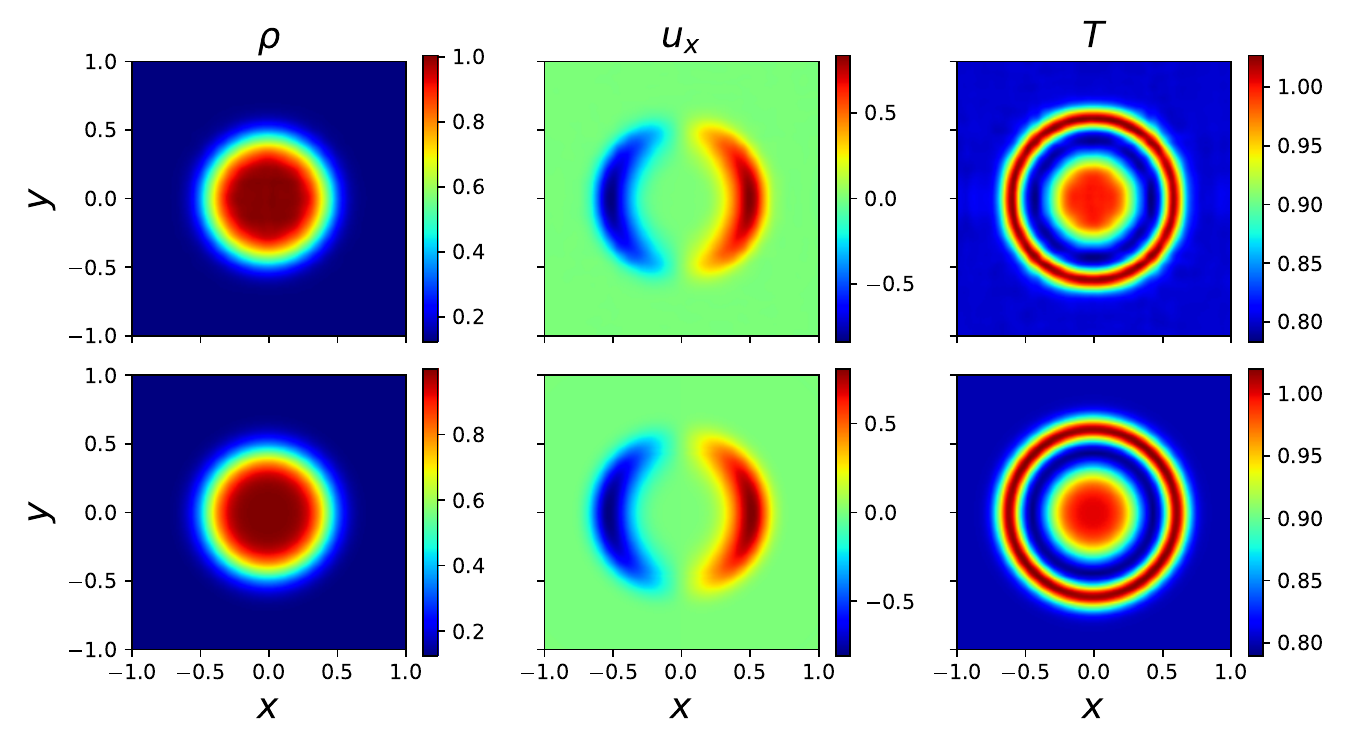}
         \caption{$\mathrm{Kn}=1.0$}
         \label{fig:2d_riemann_1.0}
     \end{subfigure}
     \begin{subfigure}[b]{0.69\textwidth}
         \centering
         \includegraphics[width=\textwidth]{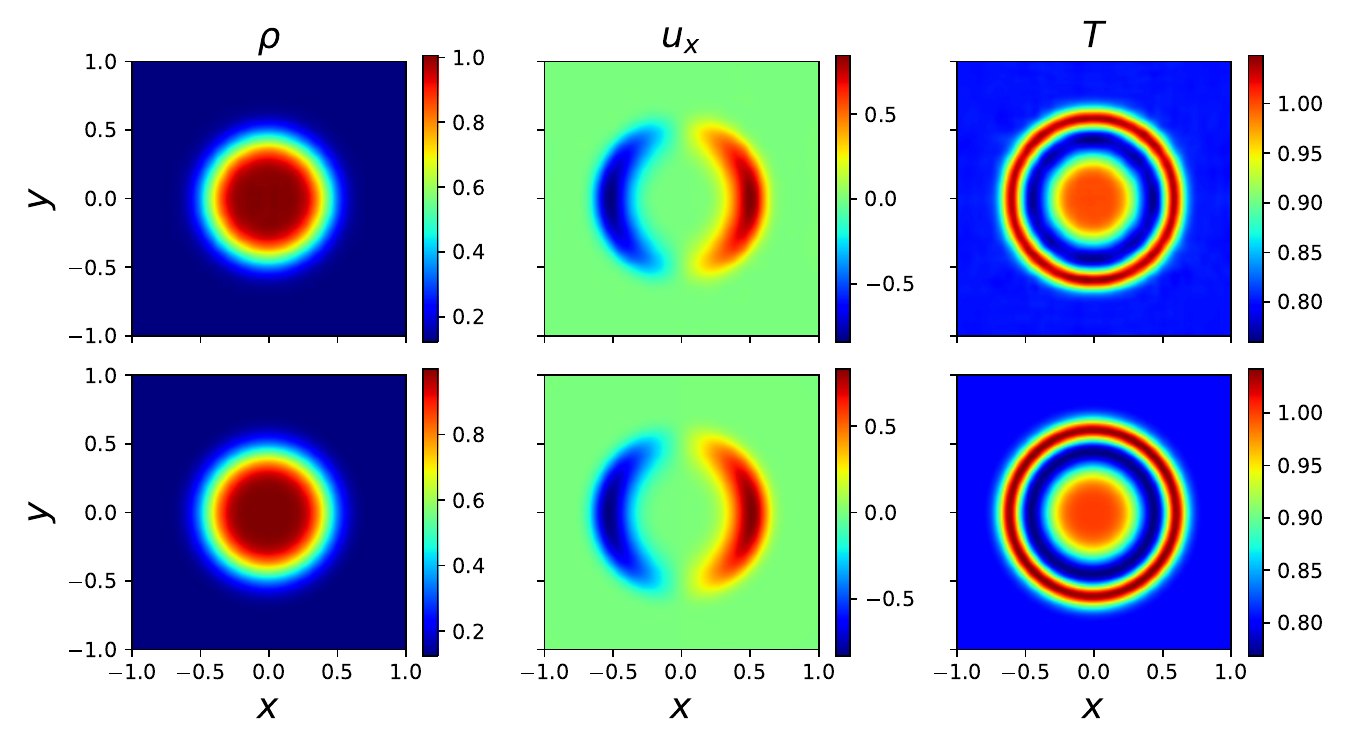}
         \caption{$\mathrm{Kn}=0.1$}
         \label{fig:2d_riemann_0.1}
     \end{subfigure}
     \begin{subfigure}[b]{0.69\textwidth}
         \centering
         \includegraphics[width=\textwidth]{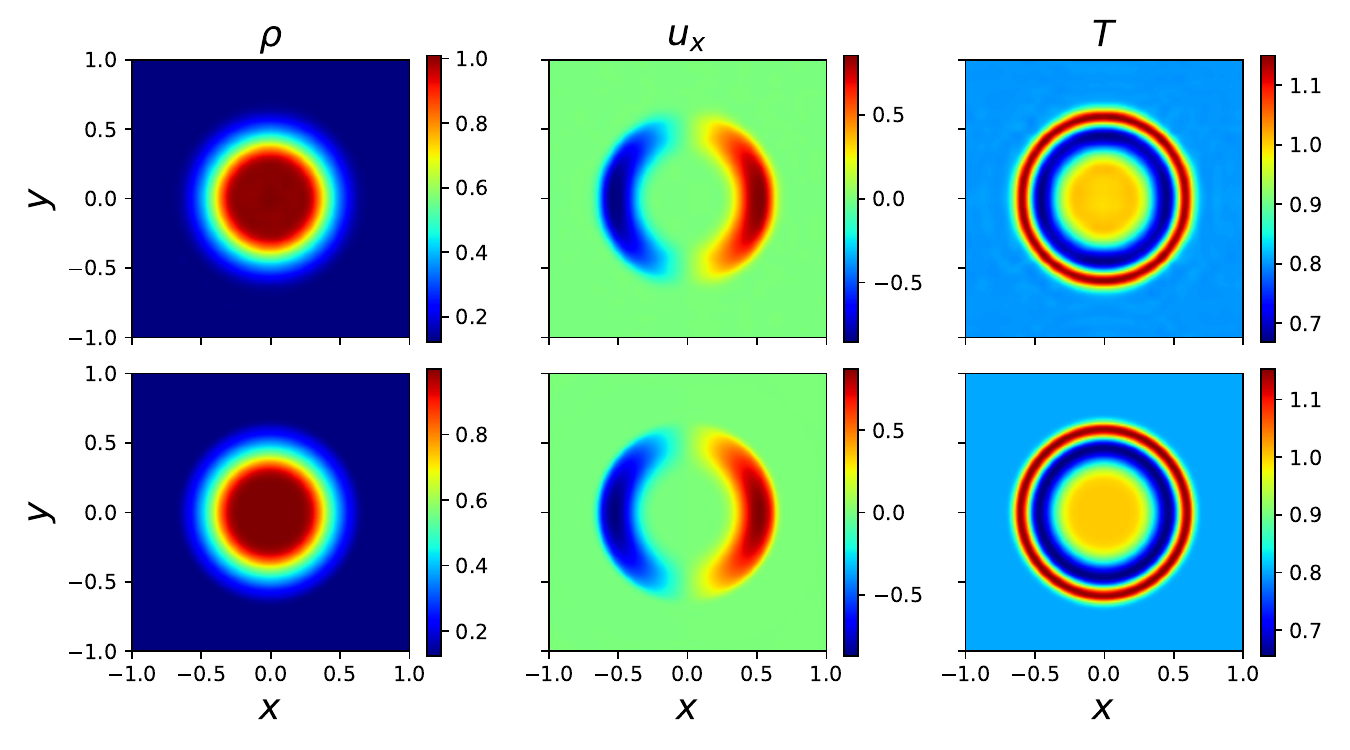}
         \caption{$\mathrm{Kn}=0.01$}
         \label{fig:2d_riemann_0.01}
     \end{subfigure}
        \caption{
            Macroscopic moments $\rho, u_x, T$ at $t=0.1$ for the 2d Riemann problem.
            Top row: neural network predictions.
            Bottom row: reference solutions.
            First column: $\rho$.
            Second column: $u_x$.
            Third column: $T$.
            Due to the symmetry, we omitted plots for $u_y$ and $u_z$.
        }
    \label{fig:2d_riemann}
\end{figure}

\begin{table}
\centering
\begin{tabular}{ |c|c|c|c| } 
 \hline
  $(t=0, t=0.1)$ & $\mathrm{Kn} = 1$ & $\mathrm{Kn} = 0.1$ & $\mathrm{Kn} = 0.01$ \\
  \hline
  $\rho$ & (4.41e-2, 2.36e-2) & (4.06e-2, 2.11e-2) & (3.59e-2, 1.78e-2) \\
  $u_x$ & (4.78e-3, 3.28e-2) & (4.07e-3, 2.89e-2) & (3.26e-3, 2.17e-2) \\
  $u_y$ & (4.34e-3, 3.14e-2) & (3.62e-3, 2.89e-2) & (3.07e-3, 2.11e-2) \\
  $T$ & (1.04e-2, 1.83e-2) & (9.86e-3, 1.61e-2) & (9.66e-3, 1.78e-2) \\
 \hline
\end{tabular}
\caption{Relative $L^2$ errors for the macroscopic moments $\rho, u_x, u_y, T$ of the 2D Riemann problem.}
\label{tab:2d_riemann}
\end{table}

\Cref{fig:2d_riemann} displays the numerical results obtained from our test.
The entire computation was completed in approximately 5 minutes and 30 seconds.
For comparison, the reference solution was generated using a grid with $N_x^2 = 160^2$ and $N_v^3 = 33^3$.
Our SPINN-BGK method successfully generated results that closely align, in qualitative terms, with the reference solution within this timeframe.
\Cref{tab:2d_riemann} presents the relative $L^2$ errors for the predicted macroscopic moments in comparison to the reference solutions.
Despite the complexity of the 2D Riemann problem, the magnitude of errors remains within the order of $O(10^{-2})$, demonstrating the effectiveness of our approach in accurately capturing the dynamics of this challenging problem.

\subsection{3D Riemann Problem}\label{sec:3d_riemann}
We finally explore the three-dimensional Riemann problem.
The spatial domain for this test is defined as $(-1, 1)^3$, where we apply a free-flow boundary condition.
Within this domain, a sphere $S$ characterized by $x^2 + y^2 + z^2 = 0.5^2$ is considered.
The initial condition is modeled as a local Maxwellian distribution with macroscopic moments $(\rho_0, {\bf u}_0, T_0)$.
Inside the sphere, the values are set to $(1, {\bf 0}, 1)$, while outside, they are $(0.375, {\bf 0}, 0.8)$.
As with the previous Riemann problems, we employ a smoothed version of the Heaviside function, $H(r^2) = (1+ \tanh(100r^2))/2$, where $r^2 = 0.5^2 - x^2 - y^2 - z^2$, to relax the discontinuity in the initial data.
The computational domain for the microscopic velocity space is established as $(-6, 6)^3$.
In each iteration, 12 collocation points per axis are sampled.

\begin{figure}
     \centering
     \begin{subfigure}[b]{0.75\textwidth}
         \centering
         \includegraphics[width=\textwidth]{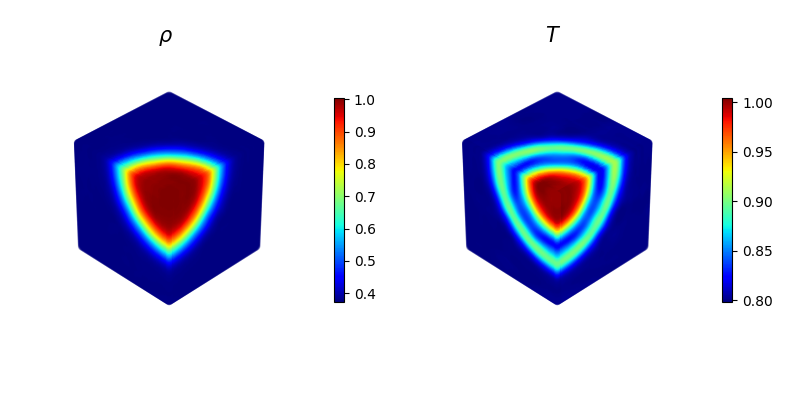}
         \caption{$\mathrm{Kn}=1.0$}
         \label{fig:3d_riemann_1.0}
     \end{subfigure}
     \begin{subfigure}[b]{0.75\textwidth}
         \centering
         \includegraphics[width=\textwidth]{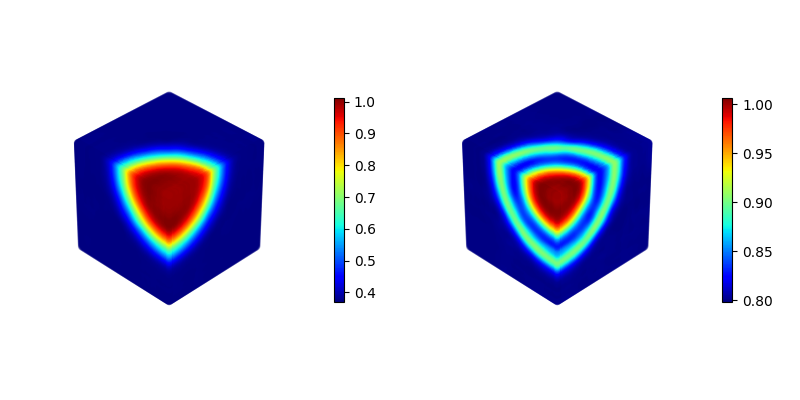}
         \caption{$\mathrm{Kn}=0.1$}
         \label{fig:3d_riemann_0.1}
     \end{subfigure}
     \begin{subfigure}[b]{0.75\textwidth}
         \centering
         \includegraphics[width=\textwidth]{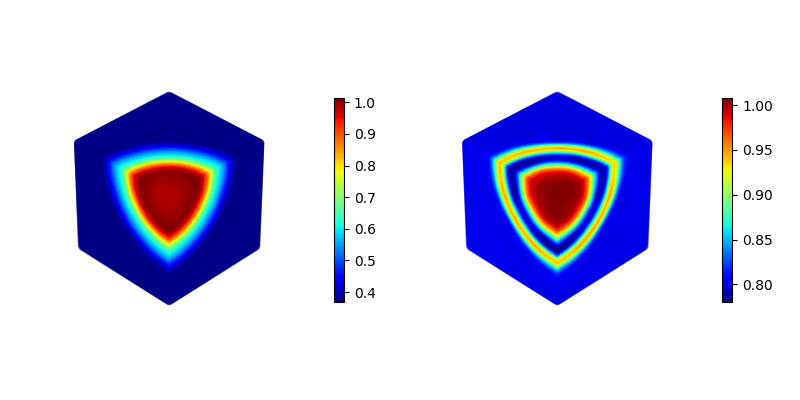}
         \caption{$\mathrm{Kn}=0.01$}
         \label{fig:3d_riemann_0.01}
     \end{subfigure}
        \caption{Predicted macroscopic moments at $t=0.1$ for the 3d Riemann problem. Left: the density $\rho$. Right: the temperature $T$.}
    \label{fig:3d_riemann}
\end{figure}

\Cref{fig:3d_riemann} displays the predicted values of $\rho$ and $T$ at $t=0.1$ within the octant $(0, 1)^2 \times (-1, 0)$.
The other octants are not displayed due to the symmetry of the problem.
The entire computation was completed in approximately 50 minutes.
Owing to the high computational costs associated with the curse of dimensionality, generating reference solutions for 3D problems was not feasible.
This challenge arises from the lack of suitable numerical schemes and the absence of open-source implementations for such complex cases.
Our research addresses this gap by successfully employing SPINN-BGK to efficiently compute a numerical solution for the 3D case on a single GPU.
This achievement demonstrates the method's capability to overcome the computational challenges inherent in high-dimensional simulations.
Based on the qualitatively accurate results obtained in \Cref{sec:1d_riemann} and \Cref{sec:2d_riemann}, we are confident that our predicted solution for the 3D case is at least qualitatively representative of the true solution.

\section{Conclusion}\label{sec:conclusion}
In this study, we have developed a separable PINN-based methodology for solving the BGK model.
The computational challenge inherent in this task arises from the high dimensionality of the equation and the necessity to compute integrals for the macroscopic moments, which significantly increase the number of network forward passes required.
To address this, we propose the SPINN-BGK, which not only reduced the number of network forward passes but also leveraged its separable structure to decrease the computational burden associated with these integrals.
A critical aspect of our approach involved addressing the potential inaccuracies that standard machine learning methods might encounter, particularly when the solution does not exhibit appropriate decay behavior for large values of $|{\bf v}|$. We effectively mitigated this issue by incorporating a Gaussian function, as defined in \cref{eq:gaussian}, into the neural network $f(\cdot, \theta_{v_p})$ for each velocity component $p \in \{x,y,z\}$.
Parameters for Gaussian functions $(\tau, \mu)$ were set to learnable parameters to automatically find optimal ones.
Moreover, we introduced a relative $L^2$ loss function, as specified in \cref{eq:relative_loss}, to dynamically adjust the weights assigned to each collocation point, enhancing the accuracy of our model.
Through a series of numerical tests, we demonstrated that SPINN-BGK is capable of producing solutions that are qualitatively in close agreement with the reference solutions for one-dimensional, two-dimensional, and three-dimensional BGK problems.

Our developments clearly point towards several extensions of great importance.
In particular, our current methodology, which relies on CP decomposition for output computation, faces challenges in efficiently processing complex geometries \cite{lee2023finite}.
Given that the velocity domain in the BGK model typically represents a product of real lines, a potential solution lies in a partial separation strategy.
This approach would involve initially segmenting the domain into time-space-velocity components and then further subdividing the velocity domain.
Such a strategy could more effectively accommodate spatial domains with intricate geometries.
In addition, the current use of backward mode AD in the PINN method for updating network parameters is memory-intensive, as it requires storing intermediate values for gradient calculation.
A re-materialization strategy could mitigate memory costs but at the expense of computational speed.
An alternative approach might involve forward mode AD, which does not necessitate storing intermediate values, although it is less efficient for gradient calculation of the loss function for parameters.
In this context, a randomized approach, as suggested in \cite{shukla2023randomized}, could offer a promising direction for research, balancing memory efficiency and computational speed.

\section*{Acknowledgement}
The work of Y. Hong was supported by Basic Science Research Program through the National Research Foundation of Korea (NRF) funded by the Ministry of Education (NRF-2021R1A2C1093579) and by the Korea government(MSIT) (RS-2023-00219980). S.Y. Cho was supported by the
National Research Foundation of Korea (NRF) grant funded by the Korea government (MSIT) (No. RS-2022-00166144).

\appendix
\section{Computation Time} \label{app:a}
\Cref{tab:computation_time} presents the computation times for solutions across various problems addressed in \Cref{sec:numerical_results}.
The second and third columns depict the computation times associated with training SPINN-BGK and generating reference solutions utilizing a high-order conservative semi-Lagrangian scheme as outlined in \cite{cho2021conservative2}, respectively.
Despite the differences in computing settings between the two methodologies, Table \ref{tab:computation_time} illustrates the effectiveness of our approach in solving BGK equations in comparison to the volumetric method.

\begin{table}
    \centering
    \begin{tabular}{|c|c|c|}
    \hline
         Problems & SPINN-BGK & Reference method in \cite{cho2021conservative2} \\
         \hline
         1d Smooth \ref{sec:1d_smooth} & 240s & 2598s \\
         1d Riemann \ref{sec:1d_riemann} & 250s & 5312s\\
         2d Smooth \ref{sec:2d_smooth} & 308s & 13764s\\
         2d Riemann \ref{sec:2d_riemann} & 331s & 36020s\\
         3d Riemann \ref{sec:3d_riemann} & 3012s & N/A \\
    \hline
    \end{tabular}
    \caption{Required computation times for \Cref{sec:numerical_results}.}
    \label{tab:computation_time}
\end{table}

\section{Computational Complexity and Knudsen Numbers}
In \Cref{fig:1d_riemann}, slight deviations are observable at both \( t = 0 \) and \( t = 0.1 \), unlike in \Cref{fig:1d_smooth}.
These deviations can be attributed to the abrupt changes in the particle density function \( f \) across \( x = 0 \), where both the magnitude and dispersion of the particle density function undergo more pronounced shifts compared to those observed in \Cref{fig:1d_smooth}.
Furthermore, for $\mathrm{Kn} = 10^0$, the deviations observed in the temperature may be linked to the distinct `shape' of the particle density function.
\Cref{fig:1d_riemann_f} displays slices of the particle density function $v_x \mapsto f$ at $t=0.1$, $x=0.15$, $v_y=v_z=0$, for $\mathrm{Kn} \in \{1.0, 0.1, 0.01\}$.
For $\mathrm{Kn} = 0.01$, the density function exhibits a bell-shaped curve.
In contrast, for $\mathrm{Kn} \in \{1.0, 0.1\}$, the function displays a bimodal distribution with two peaks, suggesting the need for a higher density of collocation points in the $v_x$ direction, thereby increasing the computational complexity.

\begin{figure}[htbp]
    \centering
    \includegraphics[width=\textwidth]{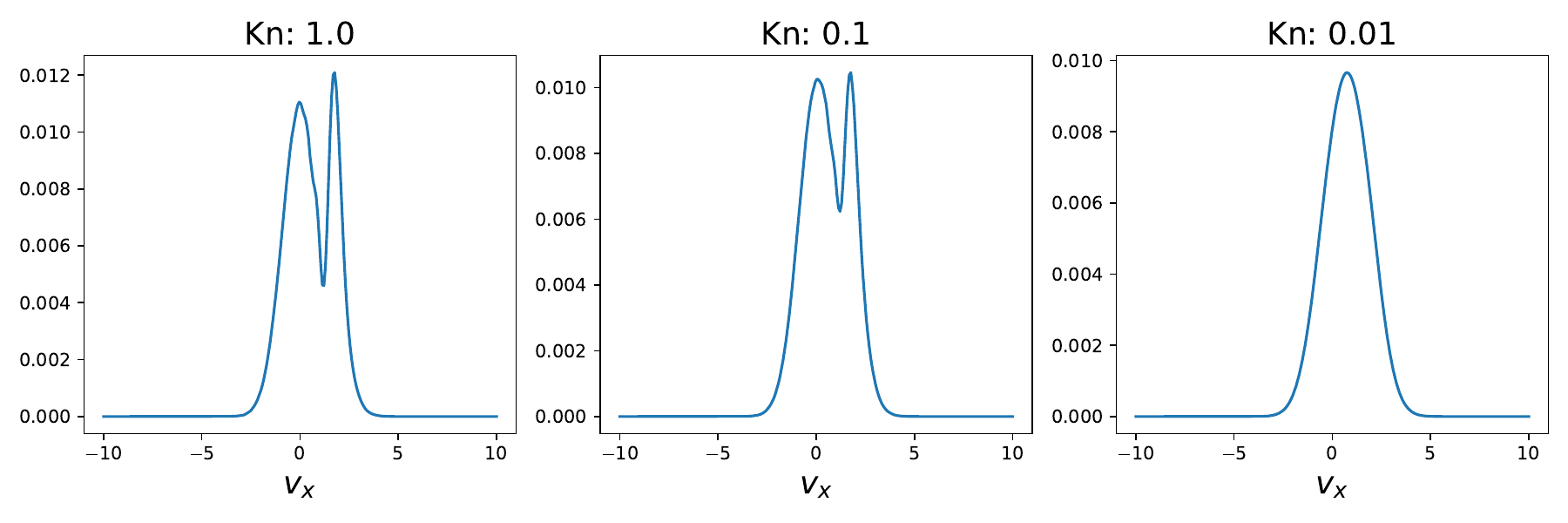}
    \caption{
        Slices of the predicted particle density $v_x \mapsto f_\theta(t, x, v_x, v_y, v_z)$ at $t=0.1$, $x=0.15$, and $v_y = v_z = 0$, for three Knudsen numbers 1.0, 0.1, and 0.01.
        Bimodal shapes would increase the computational complexity significantly, for solving the BGK model with $\mathrm{Kn} \in \{1.0, 0.1\}$.
    }
    \label{fig:1d_riemann_f}
\end{figure}

\bibliographystyle{elsarticle-harv} 
\bibliography{library}

\end{document}